\documentclass[smallextended]{svjour3}       

\smartqed  
\usepackage{graphicx}
\usepackage{amsmath, amsopn,amstext,amscd,amsfonts,amssymb}
\usepackage{dsfont} 
\usepackage{comment}
\usepackage[active]{srcltx}
\usepackage{graphicx, epsfig, subfig}
\usepackage{algorithm}

\makeatletter
\def\BState{\State\hskip-\ALG@thistlm}
\makeatother

\def\downbar#1{
\setbox10=\hbox{$#1$}
            \dimen10=\ht10 \advance\dimen10 by 2.5pt
            \ifdim \dimen10<15pt 
               \advance\dimen10 by -0.5pt
               \dimen11=\dimen10
               \advance\dimen10 by 2.5pt
               \lower \dimen11
            \else \lower \ht10 \fi
            \hbox {\hskip 1.5pt \vrule height \dimen10 depth \dp10}}
\def\upbar#1{
\setbox10=\hbox{$#1$}
            \dimen10=\ht10 \advance\dimen10 by \dp10 \advance\dimen10 by 2.5pt
            \ifdim \dimen10<15pt 
                \advance\dimen10 by 2pt \fi
            \raise 2.5pt \hbox {\hskip -1.5pt \vrule height \dimen10}}

\newcommand{\re}{\mathds{R}}
\newcommand{\te}{\mathds{T}}

\newcommand{\co}{\mathds{C}}

\newcommand{\dps}{\displaystyle}

\newcommand{\sucx}[1]{\left\{#1_n\right\}_{n\geqslant0}}
\newcommand{\suc}[1]{\left\{#1_n\right\}_{n\geqslant0}}
\newcommand{\seqdown}[3][0]{\left\{#2_#3\right\}_{#3\geqslant#1}}

\newcommand{\de}{\mathds{D}}

\newcommand{\CMcal}[1]{\mathcal{#1}}
\makeatletter\renewcommand\theenumi{\@roman\c@enumi}\makeatother

\def\cFrac#1#2{%
	\begin{array}{@{}c@{}}\multicolumn{1}{c|}{#1}\\%
		\hline\multicolumn{1}{|c}{#2}\end{array}}

\begin{document}

\title{On perturbed orthogonal polynomials on the real line and the unit circle via Szeg\H{o}'s transformation}

\titlerunning{On perturbed orthogonal polynomials on the real line and the unit circle}        

\author{K. Castillo        \and F. Marcell\'an \and J. Rivero
}


\institute{K. Castillo \at
              CMUC, Department of Mathematics, University of Coimbra,  3001-501 Coimbra, Portugal \\
              \email{kcastill@math.uc3m.es; kenier@mat.uc.pt}           
           \and
            F. Marcell\'an \at
             Instituto de Ciencias Matem\'aticas (ICMAT) and Departamento de Matem\'aticas, Universidad Carlos III de Madrid, 28911 Legan\'es, Madrid, Spain\\
             \email{pacomarc@ing.uc3m.es}
             \and
         J. Rivero \at
 Departamento de Matem\'aticas, Universidad Carlos III de Madrid, 28911 Legan\'es, Madrid, Spain\\
             \email{jorivero@math.uc3m.es}     
}

\date{Received: date / Accepted: date}

\maketitle

\begin{abstract}
By using the Szeg\H{o}'s transformation we deduce new relations between the recurrence coefficients for orthogonal polynomials on the real line and the Verblunsky parameters of orthogonal polynomials on the unit circle. Moreover, we study the relation between the corresponding $\mathcal{S}$-functions and $\mathcal{C}$-functions.\keywords{Szeg\H{o} transformation \and co-polynomials \and spectral transformations \and transfer matrices}
\subclass{42C05 \and 33C45}
\end{abstract}

\section{Introduction}

\subsection{Orthogonal polynomials on the real line and spectral transformations}
Let $d\mu$ be a non-trivial probability measure supported on $I\subseteq \re$. The sequence of polynomials $\sucx{p}$ where
$$
p_n(x)=\gamma_n x^n + \delta_n x^{n-1} + \text{(lower degree terms)}, \quad \gamma_n>0,
$$
is said to be an orthonormal polynomial sequence with respect to $d\mu$ if
$$
\int_{I} p_{n}(x)p_{m} d\mu(x) = \delta_{n,m}, \quad m\geq 0.
$$
The corresponding monic orthogonal polynomials (with leading coefficient equal to $1$) are $P_{n}(x)=p_{n}(x)/\gamma_{n}$, see \cite{C78,S05a}. These polynomials satisfy the following three-term recurrence relation
\begin{equation}\label{ttrr}
P_{n+1}(x) = (x-b_{n+1})P_n(x)-d_n P_{n-1}(x), \quad d_n\neq 0,\quad d_0=1,\quad n\geqslant 0,
\end{equation}
where the recurrence coefficients are given by
$$
b_{n} =\frac{\delta_n}{\gamma_n}-\frac{\delta_{n+1}}{\gamma_{n+1}}, \quad d_{n}=a_{n}^2, \quad a_{n}=\frac{\gamma_{n-1}}{\gamma_n}>0, \quad \quad n \geq 1.
$$
 Notice that the initial conditions $P_{-1}(x) = 0$ and $P_0(x) = 1$ hold.
 The three-term recurrence relation \eqref{ttrr} is often represented in matrix form
 $$
 x \mathbf{P}(x)=\mathbf{J} \mathbf{P}(x), \qquad \mathbf{P}=\left[P_0, P_1, \dots \right]^T,
 $$
 where $\mathbf{J}$ is a semi-infinite tridiagonal matrix
 $$
 \mathbf{J}=
 \begin{bmatrix}
 b_1 & 1 &  &  &  \\
 d_1 & b_2 & 1 & &  \\
  & d_2 & b_3 & 1 &  \\
  &  & d_3 & b_4 &  \ddots \\
  &  & & \ddots & \ddots \\
 \end{bmatrix},
 $$
 which is called the monic Jacobi matrix \cite{J48}.

The Stieltjes or Cauchy transformation of the orthogonality measure $d\mu$ is defined by
$$
S_{\mu}(x)=\int_I  \frac{d\mu(y)}{x-y}, \quad x \in \co \setminus I.
$$
It has a particular interest in the theory of {\em orthogonal polynomials on the real line} (OPRL, in short).  $S_{\mu}(x)$ admits the following series expansion
$$
S_{\mu}(x)=\sum_{k=0}^\infty \frac{u_k}{x^{k+1}},
$$
where $u_{k}$ are the moments associated with $d\mu$, i.e.,
$$
u_{k} = \int_{I} x^{k} d\mu(x).
$$
By a spectral transformation of the $\CMcal{S}$-function $S_{\mu}(x)$, we mean a new $\CMcal{S}$-function associated with a measure $d\widetilde{\mu}$, a modification of the original measure $d\mu$. We refer to pure rational spectral transformation as a transformation of $S_{\mu}(x)$ given by
\begin{equation}\label{rst}
S_\sigma(x) \ \dot{=} \ \boldsymbol{A}(x) S_\mu(x), \quad \boldsymbol{A}(x)=\begin{bmatrix} a(x) & b(x) \\ c(x) & d(x)\end{bmatrix},
\end{equation}
where $a(x)$, $b(x)$, $c(x)$, and $d(x)$ are non-zero polynomials that provide a 'true' asymptotic behavior to \eqref{rst} (see \cite{Z97}). In \eqref{rst}, we adopt the notation $\dot{=}$ introduced in \cite{S05a}, i.e., for the homography mapping
	$$
	f(x)= \frac{a(x) g(x) + b(x)}{c(x) g(x) + d(x)}, \quad a(x)d(x)-b(x)c(x)\neq 0,
	$$
	we will write
	$$
	f(x) \ \dot{=} \ \boldsymbol{A}(x) g(x).
	$$

\subsection{Orthogonal polynomials on the unit circle and spectral transformations}
Let $d\sigma$ be a non-trivial probability measure supported on the unit circle $\te = \{z\in \co:|z|=1\}$ parametrized by $z=e^{i\theta}$. There exists a unique sequence $\suc{\phi}$ of orthonormal polynomials
$$
\phi_{n}(z) = \kappa_{n} z^{n} + \text{(lower degree terms)},\quad \kappa_{n}>0,
$$
such that
\begin{equation*}
\int_{-\pi}^\pi \phi_n(e^{i \theta})\overline{\phi_m(e^{i \theta})}d\sigma(\theta)=\delta_{n,m},\quad m\geq 0. \label{ortg2}
\end{equation*}
The corresponding monic polynomials are defined by $\Phi_{n}(z)=\phi_{n}(z)/\kappa_{n}$. These polynomials satisfy the following recurrence relations (see \cite{G61,S05a,S75})
\begin{eqnarray}
\Phi_{n+1}(z) &=&  z\Phi_{n}(z) - \overline{\alpha}_{n} \Phi_{n}^{*}(z), \quad n\geq 0,\label{recurr1}\\
\Phi_{n+1}^{*}(z) &=&  \Phi_{n}^{*}(z) - \alpha_{n} z \Phi_{n}(z), \quad n\geq 0,\label{recurr2}
\end{eqnarray}
with initial condition $\Phi_0(z)=1$. The polynomial $\Phi_n^*(z)=z^n\overline{\Phi}_n(z^{-1})$ is the so-called reversed polynomial and the complex numbers $\{\alpha_n\}_{n\geq 0}$ where $\alpha_n=-\overline{\Phi_{n+1}(0)}$, are known as Verblunsky, Schur, Geronimus, or reflection parameters. Let notice that $|\alpha_{n}|<1$.  If we replace in \eqref{recurr1} the sequence $\{\alpha_n\}_{n\geq 0}$ by $\{-\alpha_n\}_{n\geq 0}$, then we obtain the sequence of second kind polynomials $\suc{\Omega}$.

The Riesz-Herglotz transform of the measure $d\sigma$ is given by
$$
F_{\sigma}(z)=\int_{-\pi}^{\pi} \frac{e^{i\theta}+z}{e^{i\theta}-z}\ d\sigma(\theta).
$$
Since $F_\sigma(0)=1$  and $\Re F_\sigma(z)>0$ on the unit open disc $\de= \{z\in \co:|z|<1\}$, $F_\sigma(z)$ is called a Carath\'eodory function \cite{G62}, or, simply, $\mathcal{C}$-function. Let $c_{k}$ be the $k$-th moment associated with the measure $d\sigma$, i.e,
$$
c_{k}=\int_{-\pi}^{\pi} e^{-ik\theta} d\sigma(\theta).
$$
$F_{\sigma}(z)$ can be written in terms of the moments $\seqdown{c}{n}$ as follows
$$
F(z) = 1+2\sum_{k=1}^{\infty}c_{k}z^k.
$$

As for the real line case, by a spectral transformation of a $\mathcal{C}$-function $F_\sigma(z)$ we mean a new $\mathcal{C}$-function associated with a measure $d\psi$, a modification of the original measure $d\sigma$. We refer to pure rational spectral transformation as a transformation of $F_\sigma(z)$ given by
\begin{equation}\label{rst2}
F_\psi(z) \ \dot{=} \ \boldsymbol{E}(z) F_\sigma(z), \quad \boldsymbol{E}(z)=\begin{bmatrix} A(z) & B(z) \\ C(z) & D(z) \end{bmatrix},
\end{equation}
where $A(z)$, $B(z)$, $C(z)$, and $D(z)$ are non-zero polynomials that provide a 'true' behavior to \eqref{rst2} around the origin (see \cite{C14}).

\subsection{Szeg\H{o} transformation and Geronimus relations}
Let us assume that the measure $d\mu$ is supported on the interval $[-1,1]$. Let introduce a measure supported on the unit circle $d\sigma$ such that
\begin{equation*}
d\sigma(\theta)=\frac{1}{2}|d\mu (\cos\theta)|.\label{sigma}
\end{equation*}
In particular, if $d\mu$ is an absolutely continuous measure, i.e., $d\mu(x)=\omega(x)dx$, we have
\[
d\sigma(\theta)=\frac{1}{2}\omega(\cos\theta)|\sin\theta|d\theta.
\]
This is the so-called Szeg\H{o} transformation of probability measures supported on $[-1,1]$ to probability measures supported on $\te$. We write the relation between $d\mu$ and $d\sigma$ through the Szeg\H{o} transformation as $\sigma= \text{Sz}(\mu)$.
Of course, under the previous considerations, we get
\[
\alpha_{n}\in(-1,1),\quad n\geqslant 0.
\]
There is a relation between the OPRL associated with a measure $d\mu$ supported on $[-1,1]$ and the OPUC associated with the measure $\sigma= \text{Sz}(\mu)$ supported on $\te$,

\begin{equation}\label{rel}
p_n(x)=\frac{\kappa_{2n}}{\sqrt{2(1-\alpha_{2n-1})}}\left(z^{-n}\Phi_{2n}(z)+z^n\Phi_{2n}(1/z)\right).
\end{equation}
From \eqref{rel} one can obtain a relation between the coefficients of the corresponding recurrence relations, see \cite{S05b},
	\begin{align}
	d_{n+1}&=\dfrac{1}{4}\left(1-\alpha_{2n-1}\right)\left(1-\alpha_{2n}^2\right)\left(1+\alpha_{2n+1}\right),\quad n\geq 0,\label{x1}\\
	b_{n+1}&=\dfrac{1}{2}[\alpha_{2n}\left(1-\alpha_{2n-1}\right)-\alpha_{2n-2}\left(1+\alpha_{2n-1}\right)]\label{x2}, \quad n\geq 0,
	\end{align}
	with the convention $\alpha_{-1}=-1$.
Notice that $b_{n}\equiv 0,\; n\geq 1$, if and only if $\alpha_{2n}=0,\; n\geq 0$.

%

There is also a relation between the $\CMcal{S}$-function and the $\CMcal{C}$-function associated with $d\mu$ and $d\sigma$, respectively, as follows

	\begin{equation*}
	F(z)=\frac{1-z^2}{2z}S(x),\label{sticar}
	\end{equation*}
	or, equivalently,
	\begin{equation*}
	S(x)=\frac{F(z)}{\sqrt{x^2-1}}, \label{sticar2}
	\end{equation*}
	with $2x=z+z^{-1}$ and $z=x-\sqrt{x^2-1}$.

The aim of this paper is to explore the connection between perturbed orthogonal polynomials on the real line and the unit circle via Szeg\H{o}'s transformation. The structure of the paper is as follows. In Section~2, we study the relation between the associated and anti-associated polynomials of order $k$ on the real line and the corresponding sequences of monic orthogonal polynomials obtained on the unit circle via the Szeg\H{o} transformation. In Section~3, we study the relation between the associated and anti-associated polynomials of order $k$ on the unit circle and the corresponding sequences of monic orthogonal polynomials obtained on the real line via the inverse of the Szeg\H{o} transformation. In Section~4, we explore the relation between the co-polynomials on the real line (on the unit circle), and the corresponding sequences of monic orthogonal polynomials obtained on the unit circle (on the real line) via Szeg\H{o} transformation. In Section~5 we investigate the relations between the symmetric polynomials on [-1,1] and sieved polynomials, and the corresponding sequence of monic orthogonal polynomials on the unit circle (on the real line) through the Szeg\H{o} transformation.

\section{Associated and anti-associated polynomials on the real line}

From the sequence of monic orthogonal polynomials $\{P_n\}_{n\geqslant0}$ we can define the sequence of associated monic polynomials of order $k$ \cite{G57}, $\{P_n^{(k)}\}_{n\geqslant0}$, $k\geqslant1$, by means of the shifted recurrence relation
\begin{equation*}
P_{n+1}^{(k)}(x)=(x-b_{n+k+1})P_n^{(k)}(x)-d_{n+k} P_{n-1}^{(k)}(x), \quad n\geqslant0,\label{recurrasoc}
\end{equation*}
with $P_{-1}^{(k)}(x) = 0$ and $P_0^{(k)}(x) =1$.

Here, we study the relation between the associated polynomials of order $k$ on the real line and the corresponding sequence of monic orthogonal polynomials obtained on the unit circle via the Szeg\H{o} transformation. We focus our attention on the resulting $\mathcal{C}$-function and the Verblumsky coefficients.

\begin{theorem}
	Let $\{\hat{\alpha}_n\}_{n \geq 0}$ be the Verblunsky coefficients for the corresponding OPUC related to the associated polynomials of order $k$ on the real line $\{P_n^{(k)}\}_{n\geqslant0}$, through the Szeg\H{o} transformation. Then, for a fixed non-negative integer $k$,
	\begin{eqnarray*}
	\hat{\alpha}_{0} &=& b_{k+1},\quad \hat{\alpha}_{1}= -1 + \frac{2 d_{k+1}}{1-\hat{\alpha}_{0}^{2}},\\
	\hat{\alpha}_{2m} &=& \frac{2 b_{m+k+1} + (1+\hat{\alpha}_{2m-1})\hat{\alpha}_{2m-2}}{1-\hat{\alpha}_{2m-1}}, \quad m\geq 1,\\
	\hat{\alpha}_{2m+1} &=& -1 + \frac{4 d_{m+k+1}}{(1-\hat{\alpha}_{2m-1})(1-\hat{\alpha}_{2m}^{2})},  \quad m\geq 1,
	\end{eqnarray*}
	with $\hat{\alpha}_{-1}=-1$.
\end{theorem}
\begin{proof}
	Let $\{\hat{b}_n\}_{n \geq 1}$ and $\{\hat{d}_n\}_{n \geq 1}$ be the recurrence coefficients for the associated polynomials of order $k$ on the real line $\{P_n^{(k)}\}_{n\geqslant0}$.
	From \eqref{x1} and \eqref{x2}, for $n\geq 0$, we have
	\begin{eqnarray*}
		\hat{\alpha}_{2n} &=& \frac{2 \hat{b}_{n+1} + (1+\hat{\alpha}_{2n-1})\hat{\alpha}_{2n-2}}{(1-\hat{\alpha}_{2n-1})},\\
		\hat{\alpha}_{2n+1} &=& -1+\frac{4 \hat{d}_{n+1}}{(1-\hat{\alpha}_{2n-1})(1-\hat{\alpha}_{2n}^2)}.
	\end{eqnarray*}
	Since $\hat{b}_n = b_{n+k}$ and $\hat{d}_n = d_{n+k}$ for all $n\geq 1$, and assume that $\hat{\alpha}_{-1}=-1$, the result follows as a consequence of straightforward computations.
%
\end{proof}
The $\mathcal{S}$-function $S^{(k)}(x)$ corresponding to the associated polynomials of order $k$, can be written as
\begin{equation} \label{Sassoc}
S^{(k)}(x)\dot{=} \boldsymbol{B}^{(k)}(x) S(x),
\end{equation}
where
$$
\boldsymbol{B}^{(k)}(x) = \begin{bmatrix}
P_{k}(x) & -P_{k-1}^{(1)}(x)\vspace{0.2cm}\\
d_{k}P_{k-1}(x) & -d_{k} P_{k-2}^{(1)}(x)
\end{bmatrix}.
$$
Then, applying the Szeg\H{o} transformation to \eqref{Sassoc} we have the following result.

\begin{theorem}\label{Szass}
	Let $\hat{F}^{(k)}(z)$ be the $\mathcal{C}$-function for the corresponding associated polynomials of order $k$ through the Szeg\H{o} transformation. Then
	\begin{equation*}
	\frac{2z}{1-z^{2}}\hat{F}^{(k)}(z)\ \dot{=}\ \boldsymbol{B}^{(k)}\left(\frac{z+z^{-1}}{2}\right) \left( \frac{2z}{1-z^{2}} F(z) \right),
	\end{equation*}
	with  $2x=z+z^{-1}$.
\end{theorem}
As a direct consequence of the above theorem, for $k=1$ we have a result proved in \cite{GM09c}.
\begin{corollary} \label{C1}
	If $F_{\Omega}(z)$ denotes the $\mathcal{C}$-function corresponding to the associated polynomials of the second kind $\Omega_{n}(z)$ on $\te$, then
	$$
	\hat{F}^{(1)}(z) = \frac{-(1-z^{2})^{2}F_{\Omega}(z)+(1-z^2)(z^2-2b_{1}z+1)}{4d_{1}z^2}.
	$$
	Notice that $F_{\Omega}(z)=\frac{1}{F(z)}$.
\end{corollary}

Let us consider a new family of orthogonal polynomials, $\{P_{n}^{(-k)}\}_{n\geqslant0}$, which is obtained by  introducing new coefficients $b_{-i}$ $(i=k-1,k-2,\ldots,0)$ on the diagonal, and $d_{-i}$ $(i=k-1,k-2,\ldots,0)$ in the lower subdiagonal of the Jacobi matrix. These polynomials are called anti-associated polynomials of order $k$, and were analyzed in \cite{RV96}.

The relation between the recurrence coefficients of the anti-associated polynomials of order $k$ on the real line and the Verblunsky coefficients for the corresponding sequence of monic orthogonal polynomials obtained on the unit circle via the Szeg\H{o} transformation can be stated as follows.

\begin{theorem}
	Let $\{\tilde{\alpha}_n\}_{n \geq 0}$ be the Verblunsky coefficients for the correspon-ding OPUC related to the anti-associated polynomials of order $k$ on the real line $\{P_n^{(-k)}\}_{n\geqslant0}$, through the Szeg\H{o} transformation. Then, for a fixed non-negative integer $k$,
	\begin{eqnarray*}
		\tilde{\alpha}_{0} &=& b_{1-k},\quad \tilde{\alpha}_{1}= -1 + \frac{2 d_{1-k}}{1-\tilde{\alpha}_{0}^{2}},\\
		\tilde{\alpha}_{2m} &=& \frac{2 b_{m-k+1} + (1+\tilde{\alpha}_{2m-1})\tilde{\alpha}_{2m-2}}{1-\tilde{\alpha}_{2m-1}}, \quad m\geq 1,\\
		\tilde{\alpha}_{2m+1} &=& -1 + \frac{4 d_{m-k+1}}{(1-\tilde{\alpha}_{2m-1})(1-\tilde{\alpha}_{2m}^{2})},  \quad m\geq 1,
	\end{eqnarray*}
	with $\hat{\alpha}_{-1}=-1$.
\end{theorem}
\begin{proof}
	Let $\{\tilde{b}_n\}_{n \geq 1}$ and $\{\tilde{d}_n\}_{n \geq 1}$ be the recurrence coefficients for the anti-associated polynomials of order $k$ on the real line $\{P_n^{(-k)}\}_{n\geqslant0}$.
	From \eqref{x1} and \eqref{x2}, for $n\geq 0$ we have
	\begin{eqnarray*}
		\tilde{\alpha}_{2n} &=& \frac{2 \tilde{b}_{n+1} + (1+\tilde{\alpha}_{2n-1})\tilde{\alpha}_{2n-2}}{(1-\tilde{\alpha}_{2n-1})},\\
		\tilde{\alpha}_{2n+1} &=& -1+\frac{4 \tilde{d}_{n+1}}{(1-\tilde{\alpha}_{2n-1})(1-\tilde{\alpha}_{2n}^2)}.
	\end{eqnarray*}
	Since $\tilde{b}_n = b_{n-k}$ and $\tilde{d}_n = d_{n-k}$ for all $n\geq 1$, and assume that $\tilde{\alpha}_{-1}=-1$. Therefore, the result follows after some computations.
\end{proof}

The $\mathcal{S}$-function $S^{(-k)}(x)$, corresponding to the anti-associated polynomials of order $k$ \cite{Z97}, can be written as
\begin{equation} \label{Santiassoc}
S^{(-k)}(x)\dot{=} \boldsymbol{B}^{(-k)}(x) S(x),
\end{equation}
where
$$
\boldsymbol{B}^{(-k)}(x) =\begin{bmatrix}
\widetilde{d}_k P_{k-2}^{(-k)}(x) & -P_{k-1}^{(-k+1)}(x)\vspace{0.2cm}\\
\widetilde{d}_k P_{k-1}^{(-k)}(x) & -P_{k}^{(-k+1)}(x)
\end{bmatrix}.
$$
From \eqref{Santiassoc} and the Szeg\H{o} transformation, we can state the analogue of Theorem~\ref{Szass}.

\begin{theorem}\label{Szantiass}
	Let $\widetilde{F}^{(-k)}(z)$ be the $\mathcal{C}$-function for the corresponding anti-associated polynomials of order $k$ through the Szeg\H{o} transformation. Then
	\begin{equation*}
	\frac{2z}{1-z^{2}}\widetilde{F}^{(-k)}(z)\ \dot{=}\ \boldsymbol{B}^{(-k)}\left(\frac{z+z^{-1}}{2}\right) \left( \frac{2z}{1-z^{2}} F(z) \right),
	\end{equation*}
	with  $2x=z+z^{-1}$.
\end{theorem}
For $k=1$, as an analog to corollaryllary \ref{C1}, we have the result proved in \cite{GM09c}.	
\begin{corollary}
	Let $\widetilde{F}_{\Omega}(z)$ be the $\mathcal{C}$-function corresponding to the anti-associated polynomials of the second kind on $\te$ through the Szeg\H{o} transformation. Then
	$$
	\widetilde{F}_{\Omega}(z) = \frac{1}{\widetilde{F}^{(-1)}(z)} = \frac{\widetilde{A}(z) F(z)+\widetilde{B}(z)}{\widetilde{D}(z)}
	$$
	where $\widetilde{A}(z)= 4\widetilde{d}_1 z^2 $,  $\widetilde{B}(z)= -(1-z^2)(z^2-2b_{1}z+1)$ and $\widetilde{D}(z)= -(1-z^2)^2$.
\end{corollary}

\section{Associated and anti-associated polynomials on the unit circle}
Let $\{\Phi_n\}_{n \geqslant0}$ be the monic orthogonal polynomial sequence with respect to a nontrivial probability measure $d\sigma$ supported on $\te$. We denote by  $\{\Phi_n^{(k)}\}_{n\geqslant0}$ be the $k$-th associated sequence of polynomials of order $k \geqslant 1$ for the monic orthogonal sequence $\{\Phi_n \}_{n \geqslant0}$, see \cite{P96}. In this case they are generated by the recurrence relation
\begin{equation*}\label{recass}
\Phi_{n+1}^{(k)}(z)=z \Phi_n^{(k)}(z)-\overline{\alpha_{n+k}}\left(\Phi_n^{(k)}(z)\right)^*, \qquad n\geq 0.
\end{equation*}
Now, we study the relation between the associated polynomials of order $k$ on the unit circle and the corresponding sequence of monic orthogonal polynomials obtained on the real line via the inverse of the Szeg\H{o} transformation. We focus our attention on the resulting $\mathcal{S}$-function and the parameters of the three term recurrence relation.

\begin{theorem}
	Let $\{\hat{b}_n\}_{n \geq 1}$ and $\{\hat{d}_n\}_{n \geq 1}$ be the recurrence coefficients for the corresponding OPRL related to the associated polynomials of order $k$ on the unit circle $\{\Phi_n^{(k)}\}_{n\geqslant0}$, through the Szeg\H{o} transformation. Then, for $k=2m-1$
	\begin{eqnarray*}
	\hat{d}_{1}&=& \frac{1+\alpha_{2m-1}}{v_{2m+1}} d_{m+1},\quad \hat{d}_{n+1}= \frac{v_{2(n+m)-1}}{v_{2(n+m)+1}} d_{n+m+1}, \quad n\geq 1,\label{x1ass1}\\
	\hat{b}_{1} &=& \alpha_{2m-1},\quad \hat{b}_{n+1} = b_{n+m+1}+ v_{2(n+m)-2}- v_{2(n+m)}, \quad  n\geq 1, \label{x2ass2}
	\end{eqnarray*}
	and for $k=2m$,
	\begin{eqnarray*}
		\hat{d}_{1}&=& \lambda d_{m+1},\quad \hat{d}_{n+1}= d_{n+m+1}, \quad n\geq 1,\label{x1ass}\\
		\hat{b}_{1} &=& \alpha_{2m},\quad \hat{b}_{n+1} = b_{n+m+1}, \quad  n\geq 1, \label{x2ass}
	\end{eqnarray*}
	with $\lambda=\frac{2}{1-\alpha_{2m-1}}$ and $v_{n}=\frac{1}{2}(1+\alpha_{n})(1-\alpha_{n-1})$.
\end{theorem}
\begin{proof}
	Let $\seqdown{\hat{\alpha}}{n}$ be the Verblunsky coefficients for the associated polynomials of order $k$ with respect to $\seqdown{\Phi}{n}$. From \eqref{x1} for $k=2m-1$ and $n=0$ we get
	\begin{eqnarray*}
		\hat{d}_{1} &=& \frac{1}{4}(1-\hat{\alpha}_{-1})(1-\hat{\alpha}_{0}^2)(1+\hat{\alpha}_{1}),\\
		&=& \frac{1}{2}(1-\alpha_{2m-1}^2)(1+\alpha_{2m}),\\
		&=& \frac{1+\alpha_{2m-1}}{v_{2m+1}}\ d_{m+1}.
	\end{eqnarray*}
	For $n\geq 1$,
	\begin{eqnarray*}
		\hat{d}_{n+1} &=& \frac{1}{4}(1-\hat{\alpha}_{2n-1})(1-\hat{\alpha}_{2n}^2)(1+\hat{\alpha}_{2n+1}),\\
		&=& \frac{1}{4}(1-\alpha_{2(n+m)-2})(1-\alpha_{2(n+m)-1}^2)(1+\alpha_{2(n+m)}),\\
		&=& \frac{v_{2(n+m)-1}}{v_{2(n+m)+1}}\  d_{n+m+1}.
	\end{eqnarray*}
	On the other hand, from \eqref{x2},
	$$
	\hat{b}_{1} = \frac{1}{2} \left[ \hat{\alpha}_{0}(1-\hat{\alpha}_{-1})-\hat{\alpha}_{-2}(1+\hat{\alpha}_{-1})\right] = \alpha_{2m-1}.
	$$
	For $n\geq 1$,
	\begin{eqnarray*}
		\hat{b}_{n+1} &=& \frac{1}{2} \left[ \hat{\alpha}_{2n}(1-\hat{\alpha}_{2n-1})-\hat{\alpha}_{2n-2}(1+\hat{\alpha}_{2n-1})\right], \\
		&=& \frac{1}{2} \left[ \alpha_{2(n+m)-1}(1-\alpha_{2(n+m)-2})-\alpha_{2(n+m)-3}(1+\alpha_{2(n+m)-2})\right],\\
		&=& -1+v_{2(n+m)-1} + v_{2(n+m)-2},\\
		&=& b_{n+m+1}+ v_{2(n+m)-2}- v_{2(n+m)}.
	\end{eqnarray*}	
	Finally, for $k=2m$, the results follow after similar computations.
\end{proof}

%

Consider the $\mathcal{C}$-function $F^{(k)}(z)$, corresponding to the associated polynomials of order $k$, given by 
\begin{equation} \label{Fassoc}
F^{(k)}(x)\dot{=} \boldsymbol{\Upsilon}^{(k)}(x) F(x)
\end{equation}
where
$$
\boldsymbol{\Upsilon}^{(k)}(x) =\begin{bmatrix}
\Phi_{k}(z)+\Phi_{k}^{*}(z) & \Omega_{k}(z)- \Omega_{k}^{*}(z)   \vspace{0.2cm}\\
\Phi_{k}(z)-\Phi_{k}^{*}(z) & \Omega_{k}(z)+ \Omega_{k}^{*}(z)
\end{bmatrix}.
$$
Then, applying the Szeg\H{o} transformation to \eqref{Fassoc} we have the following result.

\begin{theorem}\label{Szass}
	Let $\hat{S}^{(k)}(x)$ be the $\mathcal{S}$-function for the corresponding associated polynomials of order $k$ through the Szeg\H{o} transformation. Then
	\begin{equation*}
	\sqrt{x^{2}-1}\ \hat{S}^{(k)}(x)\ \dot{=}\ \boldsymbol{\Upsilon}^{(k)}\left(x-\sqrt{x^{2}-1}\right) \left( \sqrt{x^{2}-1} S(x) \right),
	\end{equation*}
	with  $z=x-\sqrt{x^{2}-1}$.
\end{theorem}
For $k=2$ we have the following result.
\begin{corollary}
	$$
	\hat{S}^{(2)}(x)\ \dot{=}\ \begin{bmatrix}
	P_{1}(x) & -1\vspace{0.2cm}\\
	(\lambda-1)(1-x^2) & (\lambda-1)(x+b_{1})
	\end{bmatrix} S(x),
	$$
	with $\lambda=\frac{2}{1-\alpha_{1}}$.
\end{corollary}
For more details about this case and the case $k=1$ see \cite{GM09d}.

Let $\{\Phi_n\}_{n \geqslant0}$ be the monic orthogonal polynomial sequence with respect to a nontrivial probability measure $d\sigma$ supported on $\te$. Let $\xi_0, \xi_1,\ldots, \xi_{k-1}$ be complex numbers with $|\xi_i| < 1$, $0\leqslant i\leqslant k-1$. We denote the anti-associated polynomials of order $k$ of $\{\Phi_n\}_{n \geqslant0}$, $\{\Phi_n^{(-k)}\}_{n\geqslant0}$,
as the sequence of monic orthogonal polynomials generated by the Verblunsky coefficients $\{\xi_i\}_{i=0}^{k-1} \bigcup \;\{\alpha_{n-k}\}_{n \geq k}$.


We now study the relation between the anti-associated polynomials of order $k$ on the unit circle and analyze the corresponding transformation obtained on the real line using the Szeg\H{o} transformation.

\begin{theorem}
	Let $\{\tilde{b}_n\}_{n \geq 1}$ and $\{\tilde{d}_n\}_{n \geq 1}$ be the recurrence coefficients for the corresponding OPRL related to the anti-associated polynomials of order $k$ on the unit circle $\{\Phi_n^{(-k)}\}_{n\geqslant0}$, through the Szeg\H{o} transformation. Let $\{\tilde{\alpha}_{n}\}_{n \geq 0}=\{\xi_n\}_{n=0}^{k-1} \bigcup \;\{\alpha_{n-k}\}_{n \geq k}$ be the Verblunsky coefficients for $\{\Phi_n^{(-k)}\}_{n\geqslant0}$ with $\tilde{\alpha}_{-1}=-1$. Then, for $k=2m-1$,
	
	\begin{equation*}
	\tilde{d}_{n+1} =  \begin{cases}
	\frac{1}{2} (1-\xi_{0}^2)(1+\xi_{1}), & n=0,\\
	\frac{1}{4}(1-\xi_{2n-1})(1-\xi_{2n}^2)(1+\xi_{2n+1}), &  1\leq n\leq m-2,\\
	\frac{1}{4}(1-\xi_{2n-1})(1-\xi_{2n}^2)(1+\alpha_{2(n-m)+2}), & n=m-1,\\
	\frac{1}{4}(1-\alpha_{2(n-m)})(1-\alpha_{2(n-m)+1}^2)(1+\alpha_{2(n-m)+2}), & n\geq m,
	\end{cases}
	\end{equation*}
	
	\begin{equation*}
	\tilde{b}_{n+1} =  \begin{cases}
	\xi_{0}, & n=0,\\
	\frac{1}{2}[(1-\xi_{2n-1})\xi_{2n}-(1+\xi_{2n-1})\xi_{2n-2}], &  1\leq n\leq m-1,\\
	\frac{1}{2}[(1-\alpha_{2(n-m)})\alpha_{2(n-m)+1}-(1+\alpha_{2(n-m)})\xi_{2n-2}], & n=m,\\
	\frac{1}{2}[(1-\alpha_{2(n-m)})\alpha_{2(n-m)+1}-(1+\alpha_{2(n-m)})\alpha_{2(n-m)-1}], & n>m.
	\end{cases}
	\end{equation*}
	For $k=2m$,
		\begin{equation*}
		\tilde{d}_{n+1} =  \begin{cases}
		\frac{1}{2} (1-\xi_{0}^2)(1+\xi_{1}), & n=0,\\
		\frac{1}{4}(1-\xi_{2n-1})(1-\xi_{2n}^2)(1+\xi_{2n+1}), &  1\leq n\leq m-1,\\
		\frac{1}{4}(1-\xi_{2n-1})(1-\alpha_{2(n-m)}^2)(1+\alpha_{2(n-m)+1}), & n=m,\\
		d_{n-m+1}, & n>m,
		\end{cases}
		\end{equation*}
		
		\begin{equation*}
		\tilde{b}_{n+1} =  \begin{cases}
		\xi_{0}, & n=0,\\
		\frac{1}{2}[(1-\xi_{2n-1})\xi_{2n}-(1+\xi_{2n-1})\xi_{2n-2}], &  1\leq n\leq m-1,\\
		\frac{1}{2}[(1-\xi_{2n-1})\alpha_{2(n-m)}-(1+\xi_{2n-1})\xi_{2n-2}], & n=m,\\
		b_{n-m+1}, & n>m.
		\end{cases}
		\end{equation*}

\end{theorem}
\begin{proof}
	From \eqref{x1}, for $n=0$, we get
	$$
	\tilde{d}_{1} = \frac{1}{4}(1-\tilde{\alpha}_{-1})(1-\tilde{\alpha}_{0}^2)(1+\tilde{\alpha}_{1}) = \frac{1}{2}(1-\xi_{0}^2)(1+\xi_{1}).
	$$
	For $1\leq n\leq m-2$,
	\begin{eqnarray*}
		\tilde{d}_{n+1} &=& \frac{1}{4}(1-\tilde{\alpha}_{2n-1})(1-\tilde{\alpha}_{2n}^2)(1+\tilde{\alpha}_{2n+1}),\\
		&=& \frac{1}{4}(1-\xi_{2n-1})(1-\xi_{2n}^2)(1+\xi_{2n+1}).
	\end{eqnarray*}
	For $n=m-1$,
	\begin{eqnarray*}
		\tilde{d}_{n+1} &=& \frac{1}{4}(1-\tilde{\alpha}_{2n-1})(1-\tilde{\alpha}_{2n}^2)(1+\tilde{\alpha}_{2n+1}),\\
		&=& \frac{1}{4}(1-\xi_{2n-1})(1-\xi_{2n}^2)(1+\alpha_{2(n-m)+2}).
	\end{eqnarray*}
	Finally, for $n \geq m$,
	\begin{eqnarray*}
		\tilde{d}_{n+1} &=& \frac{1}{4}(1-\tilde{\alpha}_{2n-1})(1-\tilde{\alpha}_{2n}^2)(1+\tilde{\alpha}_{2n+1}),\\
		&=& \frac{1}{4}(1-\alpha_{2(n-m)})(1-\alpha_{2(n-m)+1}^2)(1+\alpha_{2(n-m)+2}).
	\end{eqnarray*}
	On the other hand, from \eqref{x2}, for $n=0$,
	$$
	\tilde{b}_{1}=\frac{1}{2}[(1-\tilde{\alpha}_{-1})\tilde{\alpha}_{0}-(1+\tilde{\alpha}_{-1})\tilde{\alpha}_{-2}]=\xi_{0}.
	$$
	For $1\leq n\leq m-1$,
	\begin{eqnarray*}
		\tilde{b}_{n+1}&=&\frac{1}{2}[(1-\tilde{\alpha}_{2n-1})\tilde{\alpha}_{2n}-(1+\tilde{\alpha}_{2n-1})\tilde{\alpha}_{2n-2}],\\
		&=&\frac{1}{2}[(1-\xi_{2n-1})\xi_{2n}-(1+\xi_{2n-1})\xi_{2n-2}].
	\end{eqnarray*}
	For $n=m$
	\begin{eqnarray*}
		\tilde{b}_{n+1}&=&\frac{1}{2}[(1-\tilde{\alpha}_{2n-1})\tilde{\alpha}_{2n}-(1+\tilde{\alpha}_{2n-1})\tilde{\alpha}_{2n-2}],\\
		&=&\frac{1}{2}[(1-\alpha_{2(n-m)})\alpha_{2(n-m)+1}-(1+\alpha_{2(n-m)})\xi_{2n-2}].
	\end{eqnarray*}
	Finally, for $n>m$,
	\begin{eqnarray*}
		\tilde{b}_{n+1}&=&\frac{1}{2}[(1-\tilde{\alpha}_{2n-1})\tilde{\alpha}_{2n}-(1+\tilde{\alpha}_{2n-1})\tilde{\alpha}_{2n-2}],\\
		&=&\frac{1}{2}[(1-\alpha_{2(n-m)})\alpha_{2(n-m)+1}-(1+\alpha_{2(n-m)})\alpha_{2(n-m)-1}].
	\end{eqnarray*}
	When, for $k=2m$, the results follow in a similar way.
\end{proof}

Consider the $\mathcal{C}$-function $F^{(-k)}(z)$ corresponding to the anti-associated polynomials of order $k$, given by 
\begin{equation} \label{Fantiassoc}
F^{(-k)}(x)\dot{=} \boldsymbol{\Upsilon}^{(-k)}(x) F(x),
\end{equation}
where
$$
\boldsymbol{\Upsilon}^{(-k)}(x) = \begin{bmatrix}
\widetilde{\Omega}_k(z)+\widetilde{\Omega}_k^*(z) & \widetilde{\Omega}_k^*(z)-\widetilde{\Omega}_k(z)   \vspace{0.2cm}\\
\widetilde{\Phi}_k^*(z)-\widetilde{\Phi}_k(z) & \widetilde{\Phi}_k(z)+\widetilde{\Phi}_k^*(z)
\end{bmatrix}.
$$
Then, applying the Szeg\H{o} transformation to \eqref{Fantiassoc} we have the following result.

\begin{theorem}\label{Szassoc}
	Let $\tilde{S}^{(-k)}(x)$ be the $\mathcal{S}$-function for the corresponding anti-associated polynomials of order $k$ through the Szeg\H{o} transformation. Then
	\begin{equation*}
	\sqrt{x^{2}-1}\ \tilde{S}^{(-k)}(x)\ \dot{=}\ \boldsymbol{\Upsilon}^{(-k)}\left(x-\sqrt{x^{2}-1}\right) \left( \sqrt{x^{2}-1} S(x) \right),
	\end{equation*}
	with  $z=x-\sqrt{x^{2}-1}$.
\end{theorem}
For $k=2$, we get 
\begin{corollary}
	$$
	\tilde{S}^{(-2)}(x)\ \dot{=}\ \begin{bmatrix}
	\tilde{K}(x-\tilde{b}_{1}) & 1\vspace{0.2cm}\\
	\tilde{K}(x^2-1) & x+\tilde{b}_{1}
	\end{bmatrix} S(x),
	$$
	with $\tilde{K}=\dps \frac{1-\xi_{1}}{1+\xi_{1}}$.
\end{corollary}
An equivalent result to the previous one when $k=1$ can be found in \cite{GM09d}.

\section{Co-polynomials on the real line and the unit circle}

Let $\sucx{g}$ be a sequence of orthogonal polynomials satisfying the three term recurrence relation for OPRL with new recurrence coefficients, $\{\mathfrak{b}_n\}_{n\geq 1}$ and $\{\mathfrak{d}_n\}_{n \geq 1}$, i.e.,
		$$
		g_{n+1}(x)=(x-\mathfrak{b}_{n+1})g_n(x)-\mathfrak{d}_n g_{n-1}(x),
		$$
with initial conditions $g_{-1}(x)= 0$ and $g_0(x)= 1$,	perturbed in a (generalized) co-dilated and/or co-recursive way, namely {\em co-polynomials on the real line} (COPRL). In other words, we will consider an arbitrary single modification of the recurrence coefficients as follows
		\begin{eqnarray}
		\mathfrak{d}_{n}&=&\lambda_k^{\delta_{n,k}} d_{n}, \quad \ \ \ \ \ \ \lambda_k>0, \ \ \ \ \ \ \ \ \ \ \ \ \ \ \ \ \  \; \; \text{(co-dilated case)}\label{dil}\\
		\mathfrak{b}_{n}& =& b_{n}+\tau_{k+1} \delta_{n,k+1}, \quad \tau_{k+1} \in \re. \ \ \ \ \ \ \ \ \ \ \ \text{(co-recursive case)}\label{recu}
		\end{eqnarray}	
where $k$ is a fixed non-negative integer number, and $\delta_{nk}$ is the Kronecker delta.

The modification of the Verblunsky coefficients for the corresponding OPUC associated with the perturbed recurrence coefficients through the Szeg\H{o} transformation is shown in the following result.

\begin{theorem}
	Let $\{\widehat{\alpha}_n\}_{n \geq 0}$ be the Verblunsky coefficients for the correspon-ding OPUC,  associated with \eqref{dil} and \eqref{recu} through the Szeg\H{o} transformation. Then, for a fixed non-negative integer $k$,
	\begin{eqnarray*}
	\widehat{\alpha}_{n} &=& \alpha_{n}, \quad 0 \leq n < 2k-1,\\
	\widehat{\alpha}_{2k-1} &=& \alpha_{2k-1} + M, \\
	\widehat{\alpha}_{2k} &=& \frac{(1-\alpha_{2k-1})\alpha_{2k} + 2\tau_{k+1} + M \alpha_{2k-2}}{1-\alpha_{2k-1}-M}, \\
	\widehat{\alpha}_{2m+1} &=& -1 + \frac{4 d_{m+1}}{(1-\widehat{\alpha}_{2m-1})(1-\widehat{\alpha}_{2m}^2)}, \quad n=2m+1,\; m\geq k,\\
	\widehat{\alpha}_{2m} &=& \frac{2 b_{m+1} + (1+\widehat{\alpha}_{2m-1})\widehat{\alpha}_{2m-2}}{1-\widehat{\alpha}_{2m-1}}, \quad n=2m, \; m\geq k+1,
	\end{eqnarray*}
	where $M = \dfrac{4(\lambda_{k}-1)d_{k}}{(1-\alpha_{2k-3})(1-\alpha_{2k-2}^2)}$.
\end{theorem}

\begin{proof}
	From \eqref{x1} and \eqref{x2}, for $n\geq 0$ we have
	\begin{eqnarray*}
	\alpha_{2n} &=& \frac{2 b_{n+1} + (1+\alpha_{2n-1})\alpha_{2n-2}}{(1-\alpha_{2n-1})},\\
	\alpha_{2n+1} &=& -1+\frac{4 d_{n+1}}{(1-\alpha_{2n-1})(1-\alpha_{2n}^2)}.
	\end{eqnarray*}
	Thus, according to \eqref{dil} and \eqref{recu},
	\begin{eqnarray*}
	\widehat{\alpha}_{n} &=& \alpha_{n}, \quad 0 \leq n < 2k-1,\\
	\widehat{\alpha}_{2k-1} &=& -1+\frac{4 \lambda_{k} d_{k}}{(1-\alpha_{2k-3})(1-\alpha_{2k-2}^2)}, \\
	\widehat{\alpha}_{2k} &=& \frac{2(b_{k+1}+\tau_{k+1})+(1+\widehat{\alpha}_{2k-1})\alpha_{2k-2}}{1-\widehat{\alpha}_{2k-1}}, \\
	\widehat{\alpha}_{2m+1} &=& -1 + \frac{4 d_{m+1}}{(1-\widehat{\alpha}_{2m-1})(1-\widehat{\alpha}_{2m}^2)}, \quad n=2m+1,\; m\geq k,\\
	\widehat{\alpha}_{2m} &=& \frac{2 b_{m+1} + (1+\widehat{\alpha}_{2m-1})\widehat{\alpha}_{2m-2}}{1-\widehat{\alpha}_{2m-1}}, \quad n=2m, \; m\geq k+1,
	\end{eqnarray*}
	and the theorem follows as a consequence of straightforward computations.
\end{proof}

%

Note that the modifications \eqref{dil} and \eqref{recu} imply through the Szeg\H{o} transformation the modification of all the Verblunsky coefficients greater than $k$.

Consider the $\mathcal{S}$-function $S(x; \lambda_k, \tau_{k+1})$, associated with the COPRL \cite{CMR15}, given by
$$
S(x;\lambda_{k},\tau_{k+1})\dot{=} cof(\mathbf{M}_{k}) S(x).
$$
By applying the Szeg\H{o} transformation to this equation, we have the following result.

\begin{theorem}\label{Fs}
	Let $F(z; \lambda_k, \tau_{k+1})$ be the $\mathcal{C}$-function  associated with the perturbations \eqref{dil} and \eqref{recu} through the Szeg\H{o} transformation. Then,
	$$
	\frac{2z}{1-z^{2}} F(z;\lambda_k, \tau_{k+1})\ \dot{=} \ \text{cof} \left(\mathbf{M}_{k} \left( \frac{z+z^{-1}}{2} \right) \right) \ \left(\frac{2z}{1-z^{2}} F(z) \right),
	$$
	with  $2x=z+z^{-1}$.
\end{theorem}

For the finite composition of perturbations \eqref{dil} and \eqref{recu}, we can consider the $\mathcal{S}$-function $S(x; \lambda_m, \tau_{m+1}; \ldots; \lambda_k, \tau_{k+1})$, associated with the COPRL $P_{n}(x; \lambda_m,$ $ \tau_{m+1}; \ldots; \lambda_k, \tau_{k+1})$ \cite{CMR15}, given by
$$
S(x; \lambda_m, \tau_{m+1}; \ldots; \lambda_k, \tau_{k+1})\ \dot{=} \ cof\left( \prod_{j=m}^{k}\mathbf{M}_{j} \right) \ S(x).
$$
Then, applying the Szeg\H{o} transformation, we get the following result.

\begin{theorem}\label{Fm}
	Let $F(z;  \lambda_m, \tau_{m+1}; \ldots; \lambda_k, \tau_{k+1})$ be the $\mathcal{C}$-function  associated with the finite composition of perturbations \eqref{dil} and \eqref{recu} through the Szeg\H{o} transformation. Then,
	$$
	\frac{2z}{1-z^{2}} F(z; \lambda_m, \tau_{m+1}; \ldots; \lambda_k, \tau_{k+1})\ \dot{=} \ cof \left( \prod_{j=m}^{k}\mathbf{M}_{j} \left( \frac{z+z^{-1}}{2} \right) \right) \ \left(\frac{2z}{1-z^{2}} F(z) \right),
	$$
	with  $2x=z+z^{-1}$.
\end{theorem}

For a fixed non-negative integer number $k$, let us consider the perturbed Verblunsky coefficients $\{\beta_n\}_{n\geq0}$ given by
\begin{equation}\label{seq}
\beta_{n}=\eta_{k} \delta_{nk} + (1-\delta_{nk})\alpha_{n}. \qquad (k\text{-modification})
\end{equation}
where $\eta_{k}$ is an arbitrary complex number. In order to achieve a new sequence of Verblunsky coefficients, from now on we assume that $|\eta_k|<1$ with $\eta_k \neq \alpha_k$. We define a sequence of monic {\em co-polynomials on the unit circle} (COPUC, in short), $\{\Phi_n(\cdot; k)\}_{n\geq 0}$, those polynomials generated using $\{\beta_n\}_{n\geq0}$ through the Szeg\H{o} recurrences. Analogously, we denote by $\{\Omega_n(\cdot; k)\}_{n\geq 0}$ the corresponding second kind polynomials.

Let us consider the $\mathcal{C}$-function $F (z; l, \dots, m)$ associated with the finite composition of perturbations \eqref{seq} \cite{C14}, given by
$$
F(z; l, \dots, m) \ \dot{=} \ \prod_{j=l}^m \boldsymbol{B}_j(z) F_\sigma(z).
$$
Then, applying the Szeg\H{o} transformation, we have the following Theorem.

\begin{theorem}\label{Sm}
	Let $S(x; l,\ldots,m)$ be the $\CMcal{S}$-function for the corresponding OPRL associated with the finite composition of perturbations \eqref{seq} through the Szeg\H{o} transformation. Then,
	$$
	\sqrt{x^2-1}\ S(x; l,\ldots,m)\ \dot{=} \prod_{j=l}^{m}\boldsymbol{B}_{j}(x-\sqrt{x^2-1}) \left( \sqrt{x^2-1} S_\mu(x)\right),
	$$
	with  $z=x-\sqrt{x^{2}-1}$.
\end{theorem}


\subsection{Verblunsky coefficients and LU factorization}
If we define the sequence $\seqdown{v}{k}$,
\begin{equation}\label{vk}
v_{k} = \frac{1}{2}(1+\alpha_{k})(1-\alpha_{k-1}),
\end{equation}
then we have
\begin{eqnarray}
d_{k+1} &=& v_{2k}v_{2k+1}\label{av1},\\
b_{k+1}+1 &=& v_{2k-1} + v_{2k}\label{av2},
\end{eqnarray}
and, we can find a unique factorization
$$
\mathbf{J}+\mathbf{I}=\mathbf{L}\mathbf{U},
$$
where $\mathbf{J}$ is the Jacobi matrix associated with \eqref{ttrr}, $\mathbf{I}$ is the identity matrix, $\mathbf{L}$ is a lower bidiagonal matrix, and $\mathbf{U}$ is a upper bidiagonal matrix, with
$$
\mathbf{L} = \begin{bmatrix}
1 & \\
v_1 & 1 \\
& v_3 & 1 \\
& & \ddots & \ddots
\end{bmatrix}, \qquad
\mathbf{U} =\begin{bmatrix}
v_0 & 1 &\\
& v_2 & 1\\
&  & v_4 & 1\\
& & & \ddots & \ddots
\end{bmatrix}.
$$
Thus, from \eqref{vk}, we have
\begin{equation}\label{arec}
\alpha_{k}=-1+\frac{2v_{k}}{1-\alpha_{k-1}},
\end{equation}
or equivalently,
$$
\alpha_{k}= \dps -1 +\dps \cFrac{2v_{k}}{2}-\cFrac{2v_{k-1}}{2}-\cdots -\cFrac{2v_1}{2-v_0}.
$$
Therefore, from the a sequence $\seqdown{v}{k}$ we can determine in a very simple way the Verblunsky coefficients $\seqdown{\alpha}{k}$ for the measure $d\sigma$ supported on $\te$.

From \eqref{av1} and \eqref{av2}, we can find the sequence $\seqdown{v}{k}$ in terms of the recurrence coefficients $\seqdown[1]{b}{k}$ and $\seqdown[1]{d}{k}$, as follows
\begin{eqnarray*}
	v_{2k} &=& \dps b_{k+1}+1 -\dps \cFrac{d_{k}}{b_{k}+1}- \cdots - \dps \cFrac{d_{1}}{b_{1}+1},\\
	v_{2k+1} &=& \dps \cFrac{d_{k+1}}{b_{k+1}+1}- \dps \cFrac{d_{k}}{b_{k}+1} - \cdots - \dps \cFrac{d_{1}}{b_{1}+1},
\end{eqnarray*}
with $k\geq 0$.

If we perturb the Jacobi matrix $\mathbf{J}$ at level $k$, then we have a new sequence $\seqdown{\tilde{v}}{n}$, given by
\begin{eqnarray*}
	\tilde{v}_{2n} &=& \dps b_{n+1}+1 -\dps \cFrac{d_{n}}{b_{n}+1}-\cdots-\dps \cFrac{d_{k+1}}{b_{k+1}+\tau_{k+1}+1}-\dps \cFrac{\lambda_{k}d_{k}}{b_{k}+1}- \cdots  -\dps \cFrac{d_{1}}{b_{1}+1},\\
	\tilde{v}_{2n+1} &=& \frac{d_{n+1}}{\tilde{v}_{2n}},
\end{eqnarray*}
with $k\geq 0$.
 This can be summarized in the following Theorem.


\begin{theorem} \label{vseq}
	Let $\seqdown{\tilde{v}}{n}$ be the new sequence associated with \eqref{dil} and \eqref{recu}. Then
	\begin{eqnarray*}
	\tilde{v}_{n} &=& v_{n}, \quad 0\leq n \leq 2k-1,\\
	\tilde{v}_{2k} &=& v_{2k} + (1-\lambda_{k})v_{2k-1}+\tau_{k+1},\\
	\tilde{v}_{2m+1} &=& \frac{d_{m+1}}{\tilde{v}_{2m}}, \; \tilde{v}_{2(m+1)} = b_{m+2} + 1 -\tilde{v}_{2m+1}, \; m\geq k.
	\end{eqnarray*}
\end{theorem}
%

Therefore, as we mentioned previously, we can compute the new Verblunsky coefficients directly from the sequence $\seqdown{\tilde{v}}{n}$ as follows.
\begin{theorem} \label{Vlu}
	Let $\seqdown{\widehat{\alpha}}{n}$ be the Verblunsky coefficients for the correspon-ding OPUC 
	associated with the perturbations \eqref{dil} and \eqref{recu} through Szeg\H{o} transformation. Then,
	\begin{eqnarray*}
	\widehat{\alpha}_{n} &=& \alpha_{n}, \quad 0\leq n \leq 2k-1,\\
	\widehat{\alpha}_{2k} &=& \alpha_{2k} + \frac{2[(1-\lambda_{k})v_{2k-1}+\tau_{k+1}]}{1-\alpha_{2k-1}},\\
	\widehat{\alpha}_{n} &=& -1+\frac{2 \tilde{v}_{n}}{1-\widehat{\alpha}_{n-1}}, \quad n\geq 2k+1.
	\end{eqnarray*}
\end{theorem}
\begin{proof}
	From \eqref{arec} and Theorem~\ref{vseq}, for $n=2k$ we have
	$$
	\widehat{\alpha}_{2k} = -1 + \frac{2 \tilde{v}_{2k}}{1-\alpha_{2k-1}} = -1 + \frac{2 [v_{2k} + (1-\lambda_{k})v_{2k-1}+\tau_{k+1}]}{1-\alpha_{2k-1}}
	$$
	and the theorem follows.
\end{proof}

This is an alternative way to compute the perturbed Verblunsky coefficients through the Szeg\H{o} transformation using the LU factorization.

\section{Examples}
\subsection{Symmetric polynomials on $[-1,1]$}
Let $\seqdown{S}{n}$ be a sequence of monic symmetric polynomials orthogonal with respect to an even weight function supported on a symmetric subset of $[-1,1]$. They are generated by
$$
S_{n+1}(x) = xS_{n}(x) - d_{n} S_{n-1}(x), \quad d_n\neq 0,\quad d_0=1,\quad n\geqslant 0,
$$
with initial conditions $S_{-1}(x) = 0$ and $S_0(x) = 1$, see \cite{C78}.

Let $\seqdown{\gamma}{n}$ be the Verblunsky coefficients for the corresponding OPUC, $\{\Phi_n\}_{n \geq 0}$, related to the symmetric OPRL $\sucx{S}$, through Szeg\H{o} transformation. Then,
\begin{equation}
\gamma_{2n}=0, \quad \gamma_{2n+1}= -1 + \frac{4d_{n+1}}{1-\gamma_{2n-1}}, \quad n\geq 0\label{sym}
\end{equation}
with $\gamma_{-1}=-1$.

Indeed, from \eqref{x2}, since $b_{n}=0$ for every $n\geq 0$, we have $\gamma_{2n}=0,\ n\geq 0$. Then, from \eqref{x1}, we deduce that
	$$
	d_{n+1} = \frac{1}{4} (1-\gamma_{2n-1})(1+\gamma_{2n+1}), \quad n\geq 0,
	$$
and \eqref{sym} follows.

Let $\{\widehat{\gamma}_n\}_{n \geq 0}$ be the Verblunsky coefficients for the corresponding OPUC,  associated with \eqref{dil} through the Szeg\H{o} transformation. Then, for a fixed non-negative integer $k$,
	\begin{eqnarray*}
		\widehat{\gamma}_{2n} &=& \gamma_{2n}=0, \quad n\geq 0,\\
		\widehat{\gamma}_{2n-1} &=& \gamma_{2n-1},\quad 0\leq n <k,\\
		\widehat{\gamma}_{2k-1} &=& \gamma_{2k-1} + \frac{4(\lambda_{k}-1)d_{k}}{1-\gamma_{2k-3}}, \\
		\widehat{\gamma}_{2n+1} &=& -1 + \frac{4 d_{n+1}}{(1-\widehat{\gamma}_{2n-1})},\quad n\geq k.\\
	\end{eqnarray*}

Notice that the modification \eqref{dil} yields, from the Szeg\H{o} transformation, the modification of all odd Verblunsky coefficients greater than $k$.

\subsection{Sieved polynomials on the unit circle}
Let $d\sigma$ be a nontrivial probability measure supported on $\te$ and let $\{\Phi_n\}_{n \geq 0}$ be the corresponding OPUC. For a positive integer $\ell$ the sieved OPUC $\{\Phi^{\{\ell\}}_n\}_{n \geq 0}$ are defined as those orthogonal polynomials associated with the Verblunsky coefficients $\{\alpha^{\{\ell\}}_n\}_{n \geq 0}$ given by
\begin{equation}\label{Vs}
\alpha^{\{\ell\}}_n = \begin{cases}
\alpha_{m-1} & \text{if}\ n+1 = m \ell,\\
0 &  \text{otherwise},
\end{cases}
\end{equation}
for $n\geq 0$.  We also denote by $\sigma^{\{\ell\}}$ the nontrivial probability measure supported on  $\te$ associated with $\{\alpha^{\{\ell\}}_n\}_{n \geq 0}$. Note that $\{\Phi^{\{1\}}_n\}_{n \geq 0}$ are the polynomials $\{\Phi_n\}_{n \geq 0}$. The earliest treatment of $\{\Phi^{\{\ell\}}_n\}_{n \geq 0}$ for $\ell \geq 2$ is found in \cite{B87,MS91,IX92}. The best general reference on this subject is the work of Petronilho \cite{P08}, see also \cite{JP10}. 

Consider the case $\ell = 2$, then from \eqref{Vs} we have $\{\alpha^{\{2\}}_n\}_{n \geq 0} = \{0,\alpha_{0},0,\alpha_{1},\ldots\}$. Then we have the following result.

Let $\{b^{\{2\}}_n\}_{n \geq 1}$ and $\{d^{\{2\}}_n\}_{n \geq 1}$ be the recurrence coefficients for the corresponding OPRL, $\{P^{\{2\}}_n\}_{n \geq 0}$, related to the sieved OPUC $\{\Phi^{\{2\}}_n\}_{n \geq 0}$, through Szeg\H{o} transformation. Then,
\begin{equation}
b^{\{2\}}_{n+1} = 0, \quad d^{\{2\}}_{n+1} = \frac{1}{4} (1-\alpha_{n-1})(1+\alpha_{n}), \quad n\geq 0.\label{siev}
\end{equation}
From \eqref{x2}, since $\alpha^{\{2\}}_{2n}=0$ for every $n\geq 0$, we have $b^{\{2\}}_{n+1}=0,\ n\geq 0$. Then, from \eqref{x1}, we deduce that
	$$
	d^{\{2\}}_{n+1} = \frac{1}{4} (1-\alpha^{\{2\}}_{2n-1})(1+\alpha^{\{2\}}_{2n+1}), \quad n\geq 0.
	$$
Since $\{\alpha^{\{2\}}_n\}_{n \geq 0} = \{0,\alpha_{0},0,\alpha_{1},\ldots\}$ \eqref{siev} follows.

Let $\{\widehat{b}^{\{2\}}_n\}_{n \geq 1}$ and $\{\widehat{d}^{\{2\}}_n\}_{n \geq 1}$ be the recurrence coefficients for the corresponding OPRL associated with ($k$-modification) through the Szeg\H{o} transformation. Then, for a fixed non-negative integer $k$,
\begin{eqnarray*}
	\widehat{d}^{\{2\}}_{n+1}&=& \left(  \frac{1+\eta_{k}}{1+\alpha_{k}} \right)^{\delta_{n+1,k+1}} \left(  \frac{1-\eta_{k}}{1-\alpha_{k}} \right)^{\delta_{n+1,k+2}} d^{\{2\}}_{n+1}, \\
	\widehat{b}^{\{2\}}_{n+1}&=& 0.
\end{eqnarray*}	
Note that this $k$-modification yields, by using the Szeg\H{o} transformation, the modification of two consecutive recurrence coefficients $d^{\{2\}}_{k+1}$ and $d^{\{2\}}_{k+2}$.

\begin{acknowledgements}
The research of the first author is supported by the Portuguese Government through the FCT under the grant SFRH/BPD/101139/2014. The research of the first and second author is supported by Direcci\'on General de Investigaci\'on Cient\'ifica y T\'ecnica, Ministerio de Econom\'ia y Competitividad of Spain, grant MTM2012--36732--C03--01.
\end{acknowledgements}


\end{document}